\documentclass[a4paper,11pt]{article}
\usepackage{enumerate}
\usepackage{amsmath,amsthm,amssymb}

\title{The Brush Number of the Two-Dimensional Torus}
\author{Ta Sheng Tan\thanks{Department of Pure Mathematics and Mathematical Statistics, Centre for Mathematical Sciences, University of Cambridge, Wilberforce Road, Cambridge CB3 0WB, United Kingdom. Email: T.S.Tan@dpmms.cam.ac.uk.}}
\newtheorem{thm}{Theorem}[section]
\newtheorem{lemma}[thm]{Lemma}
\newtheorem{observation}[thm]{Observation}

\begin{document}

\maketitle

\begin{abstract}
In this paper we are interested in the brush number of a graph - a concept introduced by McKeil and by Messinger, Nowakowski and Pralat. Our main aim in this paper is to determine the brush number of the two-dimensional torus. This answers a question of Bonato and Messinger. We also find the brush number of the cartesian product of a clique with a path, which is related to the Box Cleaning Conjecture of Bonato and Messinger.
\end{abstract}

\section{Introduction}

Given a graph $G$, we consider acyclic orientations of its edges. We are interested in an acyclic orientation in which the outdegrees are `as close as possible' to the indegrees, in the following precise sense: we would like to minimise the quantity $ \sum_{v\in V(G)} \max\{0,d^+(v)-d^-(v)\}$, where $d^+(v)$ and $d^-(v)$ denote the outdegree and indegree of $v$ respectively. This quantity is called the \emph{brush number} of $G$, introduced in \cite{keil} and \cite{messinger}.
\\\\
We digress briefly to mention the original formulation, a \emph{graph cleaning} problem. The set up is as follows. Initially, all edges of a graph are considered \emph{dirty} and a fixed number of brushes are assigned to a set of vertices. At each step, a vertex $v$ is \emph{cleaned} by sending a brush along each incident dirty edge - but this is only allowed if there are at least as many brushes at $v$ as there are incident dirty edges. When a dirty edge is traversed by a brush, it is considered cleaned, and brushes cannot traverse a clean edge. A graph is cleaned when there are no more dirty edges. In this formulation, the brush number, $b(G)$ is the minimum number of brushes needed to clean a graph $G$. It is easy to see (or see later in this section) that these two formulations are the same.
\\\\
Bonato and Messinger \cite{bonato} consider the brush number of cartesian products. The cartesian product $G \times H$ of two graphs $G$ and $H$  has vertex set $V(G) \times V(H)$, and two vertices $(a,b)$ and $(c,d)$ are joined if $a=c$ and $bd\in E(H)$, or $ab\in E(G)$ and $b=d$. Based on the fact that a graph $G$ of order $m$ satisfies 
$$ 1=b(P_m) \le b(G) \le b(K_m)=\Big\lfloor \frac{m^2}{4} \Big\rfloor ,$$
which follows from results in \cite{messinger0} and \cite{messinger}, they give an elegant conjecture on bounds for the brush number of the cartesian product in general.
\\\\
\textbf{Box Cleaning Conjecture \cite{bonato}:} Given a graph $G$ of order $m$ and a graph $H$, we have
$$b(P_m \times H) \le b(G \times H) \le b(K_m \times H).$$
They gave several results towards verifying the conjecture for the case when $H$ is a path or a cycle \cite{bonato}.
\\\\
In this short paper, we find the brush number of $C_m \times C_n$ for all $m$ and $n$, and hence answer a question in \cite{bonato}. We also determine the brush number of $K_m \times P_n$; together with the results in \cite{bonato}, this verifies the Box Cleaning Conjecture for the case when $H$ is a path.
\\\\
We recall some notation from \cite{bonato}. Given a cleaning process of a graph $G$, the \emph{initial configuration} of brushes, $w_0$ is a function on $V(G)$ such that $w_0(v)$ denotes the number of brushes assigned to $v$ initially. A \emph{cleaning sequence} $\alpha = (\alpha_1,\alpha_2,\ldots,\alpha_n)$ of the vertices is the order in which vertices are cleaned in a cleaning process, that is, $\alpha_i$ is cleaned before $\alpha_j$ for $i<j$. For $\alpha_i\alpha_j\in E(G)$, we write $\alpha_i \rightarrow \alpha_j$ if  $i<j$. Note that this induces an acyclic orientation of $G$ with $\max\{0,d^+(v)-d^-(v)\}=w_0(v)$. Conversely, given an acyclic orientation of $G$, there exists an ordering of $V(G)$ such that each edge is directed from an earlier vertex to a later vertex and so induces a cleaning sequence of $G$ with $w_0(v) = \max\{0,d^+(v)-d^-(v)\}$. Hence, the graph cleaning problem is indeed equivalent to the problem mentioned at the beginning of this section. 
\\\\
We say a cleaning process or a cleaning sequence is \emph{optimal} or \emph{good} if it uses $b(G)$ brushes.

\section{Two-Dimensional Torus}

In this section, we show that the brush number for an $m$ by $n$ torus is exactly $2(m+n-2)$, i.e. $b(C_m \times C_n) = 2(m+n-2)$ for $m,n \geq 3$.
\\\\
The upper bound is easy. Indeed, we give an initial configuration using $2(m+n-2)$ brushes. Let the vertex set of $C_m$ (resp $C_n$) be $\{u_1,u_2,\ldots,u_m\}$ (resp $\{v_1,v_2,\ldots,v_n\}$) with the edge set being $\{u_iu_{i+1}:1\leq i\leq m\}$ (resp $\{v_iv_{i+1}:1\leq i\leq n\}$), where we take $u_{m+1} = u_1$ (resp $v_{n+1} = v_1$). The initial configuration for $C_m \times C_n$ is as follows.
\begin{equation*}
 w_0(u_i,v_j)=\left\{
\begin{array}{rl}
 4 & \text{if } i=j=1,\\
 2 & \text{if } i=1,2\leq j\leq n-1 \text{ or } 2\leq i \leq m-1,j=1,\\
 0 & \text{otherwise}.
\end{array} \right.
\end{equation*}
\noindent
The total number of brushes used is $2(m+n-2)$ and it is easy to see that this initial configuration suffices to clean $C_m \times C_n$.
\\\\
For the lower bound, we use induction on $m+n$. The base case would be to show $b(C_3 \times C_n)=2n+2$ for $n\geq 3$. The proof of this can be found in \cite{bonato}, but for the sake of completeness, we will include a proof here. 

\begin{lemma}\label{lemma}
For $n\geq 3$, $b(C_3 \times C_n)=2n+2$.
\end{lemma}
\begin{proof}
We can express $C_3 \times C_n$ as $n$ copies of $C_3$. Suppose $C_3 \times C_n$ is cleaned using $b(C_3 \times C_n)$ brushes. Let $w_i$ be the first vertex cleaned in the $i$th copy of $C_3$. Suppose there are $k$ vertices in $\{w_1,w_2,\ldots,w_n\}$ that received two brushes from the neighbouring copies of $C_3$ before it is cleaned, then there must be at least $k$ other vertices in $\{w_1,w_2,\ldots,w_n\}$ sending two brushes to both the neighbouring copies of $C_3$ when it is cleaned.
\\\\
This gives a lower bound of the number of brushes in the initial configuration of $\{w_1,w_2,\ldots,w_n\}$, which is $4k+2(n-2k)=2n$. Now, look at the set of the second vertex cleaned in each copy of $C_3$ and let $u$ be the first vertex cleaned in this set. It is not hard to see that either $u$ has two brushes in the initial configuration or there are two extra brushes in $\{w_1,w_2,\ldots,w_n\}$. This completes the proof.
\end{proof}

\noindent
Now we prove the lower bound for the brush number of $C_m \times C_n$ for general $m$ and $n$.

\begin{thm}
For $m,n\geq 3$, the number of brushes needed to clean $C_m \times C_n$ is at least $2(m+n-2)$.
\end{thm}

\begin{proof}
 We use induction on $m+n$. The theorem is true for $m=3$ or $n=3$ by Lemma~\ref{lemma}. So assume $m,n\geq 4$ and suppose the theorem is true for $C_{m-1} \times C_n$ and $C_m \times C_{n-1}$.
\\\\
Given a cleaning sequence of $C_m \times C_n$, $\alpha$ using $b(C_m \times C_n)$ brushes with initial configuration $w_0$, we claim that we can combine any two consecutive rows (or columns) to provide a cleaning sequence for $C_{m-1} \times C_n$ (or $C_m \times C_{n-1}$) without using any extra brushes.
\\\\
Without loss of generality, assume we are combining the last two rows. We provide an initial configuration for $C_{m-1} \times C_n$, $w_0^{'}$ from $w_0$ as follows.
\begin{equation*}
 w_0^{'}(u_i,v_j)=\left\{
\begin{array}{ll}
 w_0(u_i,v_j) + w_0(u_{i+1},v_j) & \text{if } i=m-1\\
 w_0(u_i,v_j) & \text{otherwise}
\end{array} \right.
\end{equation*}

\noindent
It is easy to see that this initial configuration suffices to clean $C_{m-1} \times C_n$. Indeed, we can clean $C_{m-1} \times C_n$ by going along $\alpha$ and whenever we come to a vertex in the last two rows of $C_m \times C_n$, we clean the corresponding vertex in $C_{m-1} \times C_n$ if it has not been cleaned.
\\\\
This shows that the brush number of $C_m \times C_n$ is at least the brush number of $C_{m-1} \times C_n$. To prove the theorem, we need to show that we can take away two brushes during this combining process while still leaving enough to clean $C_{m-1} \times C_n$. We claim this can be done by picking two `correct' rows (or columns).
\\\\
It is clear that the maximum number of brushes at a vertex at any time is 4 and that the sum of the number of brushes at two adjacent vertices is at most 6 in any good cleaning sequence of 2-dimensional torus. So we can go along the cleaning sequence $\alpha$ and locate the first vertex with less than 4 brushes in the initial configuration. This vertex must have an earlier neighbour. The two rows (or columns) where each contains one of these two vertices are the `correct' rows (or columns). Indeed, if this vertex has 2 brushes in the initial configuration, then by combining the two `correct' rows (or columns), the corresponding vertex would have 6 brushes in the new initial configuration; otherwise the vertex does not have any brushes initially and this implies that it has two earlier neighbours and by combining the two `correct' rows (or columns), there would be two adjacent vertices each with 4 brushes in the new initial configuration. This completes the proof.
\end{proof}

\section{Cleaning $K_m \times P_n$}

In this section, we show that the brush number of $K_m \times P_n$ is exactly $n\big\lfloor\frac{m^2}{4}\big\rfloor$  when $m$ is even and $n\big\lfloor\frac{m^2}{4}\big\rfloor + 1$ when $m$ is odd.
\\\\
We will assume $m$ is even throughout this section and the case of $m$ being odd is similar and omitted.
\\\\
For the upper bound, we will give an initial configuration using $n(\frac{m^2}{4})$ brushes. It is natural to let the vertex set of $K_m \times P_n$ be $\{(x_i,y_j):1\leq i\leq m, 1\leq j\leq n\}$. The initial configuration for $K_m \times P_n$ is as follows.
\begin{equation*}
 w_0(x_i,y_j)=\left\{
\begin{array}{ll}
 \max\{m+2-2i,0\} & \text{if } j=1,\\
 \max\{m+1-2i,0\} & \text{if } 1<j<n,\\
 \max\{m-2i,0\} & \text{if } j=n.
\end{array} \right.
\end{equation*}

\noindent
The total number of brushes used is $n(\frac{m^2}{4})$ and it is easy to see that this initial configuration suffices to clean $K_m \times P_n$.
\\\\
For the lower bound, we use induction on $n$. The idea of the proof is to find a new configuration for $K_m \times P_{n-1}$ from the optimal cleaning process of $K_m \times P_n$ by deleting one of the end copies of $K_m$ and modifying the configuration of brushes in the adjacent copy of $K_m$. More precisely, for each vertex in the second copy of $K_m$, we want to add brushes to it or take away brushes from it depending on whether it has brushes in the initial configuration, whether its neighbour in the first copy of $K_m$ has brushes in the initial configuration and also whether it is cleaned before its neighbour in the first copy of $K_m$. There are two choices for each of these three conditions, hence we can divide adjacent pair of vertices into eight classes.
\\\\
We will first give a lemma to show that the base case holds, and then we will prove the inductive step, where we use very similar arguments to those in the proof of the base case. The following trivial observation turns out to be very useful in our proof.

\begin{observation} \label{obs}
In any cleaning sequence, there cannot be a set of four vertices, $\{a,b,c,d\}$ such that $a\rightarrow b$, $b\rightarrow c$, $c\rightarrow d$ and $d\rightarrow a$.
\end{observation}

\noindent
We will show the base case, that is the brush number of $K_m \times P_2$. 

\begin{lemma} \label{lemma2}
 $b(K_m \times P_2) \geq \frac{m^2}{2}$.
\end{lemma}

\begin{proof}
 Suppose $K_m \times P_2$ is cleaned using $b(K_m \times P_2)$ brushes with initial configuration $w_0$. Let $u_i$ be the $i$th vertex cleaned in the first copy of $K_m$ and $v_j$ be the $j$th vertex cleaned in the second copy of $K_m$. In the proof, $u$'s will always refer to vertices in the first copy of $K_m$ and $v$'s will always refer to vertices in the second copy of $K_m$.
\\\\
We can split adjacent pair of vertices $(u,v)$ into eight classes as follows.
\begin{align*}
A &= \{(u,v):w_0(u)>0,u\rightarrow v,w_0(v)>0\},\quad |A|=a\\
B &= \{(u,v):w_0(u)>0,u\rightarrow v,w_0(v)=0\},\quad |B|=b\\
C &= \{(u,v):w_0(u)>0,u\leftarrow v,w_0(v)>0\},\quad |C|=c\\
D &= \{(u,v):w_0(u)>0,u\leftarrow v,w_0(v)=0\},\quad |D|=d\\
E &= \{(u,v):w_0(u)=0,u\rightarrow v,w_0(v)>0\},\quad |E|=e\\
F &= \{(u,v):w_0(u)=0,u\rightarrow v,w_0(v)=0\},\quad |F|=f\\
G &= \{(u,v):w_0(u)=0,u\leftarrow v,w_0(v)>0\},\quad |G|=g\\
H &= \{(u,v):w_0(u)=0,u\leftarrow v,w_0(v)=0\},\quad |H|=h
\end{align*}

\noindent
Deleting the first copy of $K_m$ and modifying the second copy of $K_m$ according to which type each pair of vertices belongs to, we can give a configuration to clean $K_m$. It is not too hard to see that if $w_0(u_i)>0$, then $i\leq \frac{m}{2}$ and if $w_0(u_i)=0$, then $i\geq \frac{m}{2}$. The same is true for $v_j$. In fact, for $u_i$ with non-empty brushes in the initial configuration, $w_0(u_i) = m-2i+2$ if it is in $A$ or $B$ and $w_0(u_i) = m-2i$ otherwise. Therefore, there are $\frac{m^2-2m}{4}+2a+2b$ brushes in the first copy of $K_m$. To be able to clean $K_m$ after deleting the first copy of $K_m$, we shall modify the configuration in the second copy of $K_m$ accordingly. Vertices in $A$ and $E$ would each need an extra brush while every vertex in $C$ and $G$ has one extra brush. The configuration of the vertices in the other classes need not be changed, unless $w_0(v_{\frac{m}{2}})=0$, in which case we need one extra brush. This give a lower bound for $b(K_m \times P_2)$,

$$b(K_m \times P_2) - \frac{m^2-2m}{4}-a-2b + e - c - g + \delta_{w_0(v_{\frac{m}{2}})=0}\ge b(K_m),$$
where $\delta_{w_0(v_{\frac{m}{2}})=0}$ is 1 if $w_0(v_{\frac{m}{2}})=0$ and 0 otherwise. As we know $b(K_m)=\frac{m^2}{4}$, we are left to show $a+2b-e+c+g-\delta_{w_0(v_{\frac{m}{2}})=0}\ge \frac{m}{2}$.
\\\\
Give a pair of vertices in $D$, $(u,v)$ and a pair of vertices in $E$, $(u',v')$, we can see that $u \rightarrow u'$, $u' \rightarrow v'$, $v' \rightarrow v$, $v \rightarrow u$, contradicting Observation~\ref{obs}. So we can, without loss of generality, assume $e=0$.
\\\\
Vertices in the second copy of $K_m$ that are in $A$, $C$, and $G$ are those with non-empty brushes in the initial configuration. Hence we have $a+c+g = \frac{m}{2} - \delta_{w_0(v_{\frac{m}{2}})=0}$. So we are left to show $2b \ge 2\delta_{w_0(v_{\frac{m}{2}})=0}$. We are done, unless $b=0$ and $w_0(v_{\frac{m}{2}})=0$. In which case $v_{\frac{m}{2}}$ is in $F$, as we are assuming $B=E=\empty$. Given a pair of vertices in $D$, then together with the pair of vertices in $F$ that contains $v_{\frac{m}{2}}$, we have a contradiction to Observation~\ref{obs}. So $d=0$.
\\\\
Now $a+c=\frac{m}{2}-\delta_{w_0(u_{\frac{m}{2}})=0}$, where $\delta_{w_0(u_{\frac{m}{2}})=0}$ is 1 if $w_0(u_{\frac{m}{2}})=0$ and 0 otherwise. So we have $g=\delta_{w_0(u_{\frac{m}{2}})=0}-1$, which implies $g=0$ and $w_0(u_{\frac{m}{2}})=0$. So $u_{\frac{m}{2}}$ must be in $H$. The pair of vertices in $F$ that contains $v_\frac{m}{2}$ and the pair of vertices in $H$ that contains $u_\frac{m}{2}$ again give a contradiction to Observation~\ref{obs}. This completes the proof.
\end{proof}

\noindent
Now we prove our theorem on the exact brush number of $K_m \times P_n$.

\begin{thm} \label{theorem}
For even $m\ge 2$ and $n\ge 2$, $b(K_m \times P_n) = n(\frac{m^2}{4})$.
\end{thm}

\begin{proof}
Suppose $n\ge 3$. Using similar case analysis in the proof of the Lemma~\ref{lemma2}, we will show that $b(K_m \times P_n) \ge b(K_m \times P_{n-1}) + \frac{m^2}{4}$.
\\\\
Suppose $K_m \times P_n$ is cleaned using $b(K_m \times P_n)$ brushes with initial configuration $w_0$. Let $u_i$ be the $i$th vertex cleaned in the first copy of $K_m$ and $v_j$ be the $j$th vertex cleaned in the second copy of $K_m$. In the proof, $u$'s will always refer to vertices in the first copy of $K_m$ and $v$'s will always refer to vertices in the second copy of $K_m$.
\\\\
Like before, we split adjacent pair of vertices $(u,v)$ into eight classes as follows.

\begin{align*}
A &= \{(u,v):w_0(u)>0,u\rightarrow v,w_0(v)>0\},\quad |A|=a\\
B &= \{(u,v):w_0(u)>0,u\rightarrow v,w_0(v)=0\},\quad |B|=b\\
C &= \{(u,v):w_0(u)>0,u\leftarrow v,w_0(v)>0\},\quad |C|=c\\
D &= \{(u,v):w_0(u)>0,u\leftarrow v,w_0(v)=0\},\quad |D|=d\\
E &= \{(u,v):w_0(u)=0,u\rightarrow v,w_0(v)>0\},\quad |E|=e\\
F &= \{(u,v):w_0(u)=0,u\rightarrow v,w_0(v)=0\},\quad |F|=f\\
G &= \{(u,v):w_0(u)=0,u\leftarrow v,w_0(v)>0\},\quad |G|=g\\
H &= \{(u,v):w_0(u)=0,u\leftarrow v,w_0(v)=0\},\quad |H|=h
\end{align*}

\noindent
Deleting the first copy of $K_m$ and modifying the second copy of $K_m$ accordingly, we give a new initial configuration to clean $K_m \times P_{n-1}$. As before, if $w_0(u_i)>0$, then $i\leq \frac{m}{2}$ and if $w_0(u_i)=0$, then $i\geq \frac{m}{2}$. As for vertices in second copy of $K_m$, if $w_0(v_i)>0$, then $i\leq \frac{m}{2}+1$ and if $w_0(v_i)=0$, then $i\geq \frac{m}{2}$. Also, for $u_i$ with non-empty brushes in the initial configuration, $w_0(u_i) = m-2i+2$ if it is in $A$ or $B$ and $w_0(u_i) = m-2i$ otherwise. The number of brushes in the first copy of $K_m$ is then $\frac{m^2-2m}{4}+2a+2b$. After deleting the first copy of $K_m$, we need an extra brush for each vertex in $A$ and $E$. We can also take away a brush from each vertex in $C$ and $G$. Notice that every vertex in second copy of $K_m$ has degree $m+1$ and so even if $w_0(v_{\frac{m}{2}})=0$, we do not need an extra brush for it. This gives a lower bound for $b(K_m \times P_n)$,

$$b(K_m \times P_n) - \frac{m^2-2m}{4}-a-2b + e - c - g \ge b(K_m \times P_{n-1}).$$
We are left to show $a+2b-e+c+g\ge \frac{m}{2}$.
\\\\
A pair of vertices in $D$ and a pair of vertices in $E$ will again give a contradiction to Observation~\ref{obs}. So we have two cases to check.
\\\\
\textbf{Case 1: $e=0$.}

\noindent
Vertices in the second copy of $K_m$ that are in $A$, $C$, and $G$ are those with non-empty brushes in the initial configuration. Hence we have $a+c+g \ge \frac{m}{2} - \delta_{w_0(v_{\frac{m}{2}})=0}$. We are now done, unless $a+c+g=\frac{m}{2}-1$, $b=0$ and $w_0(v_{\frac{m}{2}})=0$, in which case $v_{\frac{m}{2}}$ is in $F$. Now, we must have $d=0$ as otherwise we will again contradict Observation~\ref{obs} from a pair of vertices in $D$ and the pair of vertices in $F$ that contains $v_{\frac{m}{2}}$.
\\\\
Exactly like before, we now have $a+c=\frac{m}{2}-\delta_{w_0(u_{\frac{m}{2}})=0}$ and so $g=\delta_{w_0(u_{\frac{m}{2}})=0}-1$, which implies $g=0$ and $w_0(u_{\frac{m}{2}})=0$. So $u_{\frac{m}{2}}$ must be in $H$. The pair of vertices in $F$ that contains $v_\frac{m}{2}$ and the pair of vertices in $H$ that contains $u_\frac{m}{2}$ contradict Observation~\ref{obs}. This completes the proof for Case 1.
\\\\
\textbf{Case 2: $d=0,e>0$.}

\noindent
In this case, we have $a+c+e+g\le \frac{m}{2} + \delta_{w_0(v_{\frac{m}{2}+1})=0}$, where $\delta_{w_0(v_{\frac{m}{2}+1})=0}$ is 1 if $w_0(v_{\frac{m}{2}+1})=0$ and 0 otherwise. So we are left to show $2a+2b+2c+2g\ge m+\delta_{w_0(v_{\frac{m}{2}+1})=0}$.
\\\\
Since $d=0$, we have $a+b+c=\frac{m}{2}-\delta_{w_0(u_{\frac{m}{2}})=0}$ and we are done unless $2g<2\delta_{w_0(u_{\frac{m}{2}})=0}+\delta_{w_0(v_{\frac{m}{2}+1})=0}$. We will show that this is not possible.
\\\\
We can see that $w_0(u_{\frac{m}{2}})=0$ only when $u_{\frac{m}{2}}$ is in $G$ or $H$, but if it is in $H$, then together with a pair of vertices in $E$ contradicts Observation~\ref{obs}. So $w_0(u_{\frac{m}{2}})=0$ implies $u_{\frac{m}{2}}$ is in $G$. Similarly, $w_0(v_{\frac{m}{2}+1})=0$ implies $v_{\frac{m}{2}+1}$ is in $G$. So $2g<2\delta_{w_0(u_{\frac{m}{2}})=0}+\delta_{w_0(v_{\frac{m}{2}+1})=0}$ can only hold when $g=1$ and $(u_{\frac{m}{2}},v_{\frac{m}{2}+1})$ is in $G$. This is not possible as together with a pair of vertices in $E$, we have a contradiction to Observation~\ref{obs}. This completes the proof for Case 2.
\\\\
Together with Lemma~\ref{lemma2}, and the initial configuration given at the beginning of this section, the proof of the theorem is now complete. 

\end{proof}

\section{Remarks}

As we can see in the proof of Lemma~\ref{lemma2} and Theorem~\ref{theorem}, the order of the vertices being cleaned can be determined from the number of brushes in the initial configuration of any optimal cleaning process, and this plays an important role in our proof. So this method could only work on very structured graphs. For example, it should be possible to use the same method to show that $b(K_m \times C_n) \ge b(K_m \times P_{n-1})+ \big\lfloor \frac{m^2}{4} \big\rfloor + 2-\delta_{m,odd}$, and hence give a different proof to \cite{bonato} for $b(K_m \times C_n)=\big\lfloor \frac{m^2}{4} \big\rfloor + 2$. But instead of splitting pairs of vertices into eight classes, we would need to split triplets of vertices into 32 classes.

\subsection*{Acknowledgement}
 
The author would like to thank Imre Leader for his invaluable comments.

\end{document}